\documentclass[12pt,a4paper,twoside,british,english,openright]{article}
\usepackage[T1]{fontenc}
\usepackage[latin1]{inputenc}
\usepackage{fancyhdr}
\usepackage{calc}
\usepackage{mathtools}
\usepackage{amsmath}
\usepackage{amsthm}
\usepackage{amssymb}
\usepackage{mathdots}
\usepackage{undertilde}
\usepackage{setspace}
\usepackage{esint}
\makeatletter

\theoremstyle{plain}

\newtheorem{theorem}{Theorem}[section]

\newtheorem*{remark*}{Remark}

\usepackage[english]{babel}

\usepackage{babel}

\addto\captionsbritish{}
\addto\captionsbritish{}
\addto\captionsenglish{}
\addto\captionsenglish{}

\makeatother

\usepackage{babel}
\addto\captionsbritish{}
\addto\captionsbritish{}
\addto\captionsenglish{}
\addto\captionsenglish{}

\begin{document}
\title{\textbf{Semialgebraic and Continuous Solution }\\
 \textbf{ of Linear Equation with }\\
 \textbf{ Semialgebraic Coefficients}}
\author{Marcello Malagutti}
\date{}
\maketitle

\begin{abstract}
Starting from the results of Charles Fefferman and Janos Kollár in \emph{Continuous Solutions of Linear Equations} \cite{FK}, we adopt a new approach based on Fefferman's techniques of Glaeser refinement to show a more general result than the one proved by Kollár by using techniques from algebraic geometry. Considering a system of linear equations with semialgebraic (not only polynomial as in \cite{FK}) coefficients on $\mathbb{R}^{n}$, we get a necessary and sufficient condition for the existence of a continuous and semialgebraic solution on $\mathbb{R}^{n}$. This is different from what Fefferman and Luli obtained in \emph{Semialgebraic Sections Over the Plane} \cite{FL-2} since they stated their result for solutions of regularity $C^m$ on the plane $\mathbb{R}^2$. More in depth, we prove that a continuous and semialgebraic solution on $\mathbb{R}^{n}$ exists if and only if there is a continuous solution i.e., if the Glaeser-stable bundle associated to the system has no empty fiber.
\end{abstract}

\maketitle



\section{Introduction}
\label{intro}
This work deal with the open problem of obtaining by analytical techniques necessary and sufficient 
conditions for the existence of a  $C^m$ and semialgebraic solution of a system of linear equation 
with semialgebraic coefficients. In case $m=0$ of a system with polynomial coefficients, the problem 
was solved by Fefferman-Kollár \cite{FK} and Kollár \cite{K} using algebraic techniques for systems 
with polynomial coefficients. In this work, 
by a new approach based on Fefferman\'s analytic techniques 
of Glaeser\'s refinements, we solve the problem for the case of a $C^0$ and semialgebraic solution on a 
general $n$-dimension space $\mathbb{R}^{n}$, extending Fefferman- Kollár\'s result to the case of a system 
with semialgebraic (not only polynomial as in \cite{FK}) coefficients. 

\noindent Let us go throught a deeper explanation of our work and the context in which it is developed. 

C. Fefferman proved in \cite{FK}, by means of analysis techniques,
a necessary and sufficient condition for the existence of a continuous
solution $(\phi_{1},\ldots,\phi_{s})$ of the system
\begin{equation}
\phi=\sum\limits _{i=1}^{s}\phi_{i}f_{i}\label{eq:System_1}
\end{equation}
given the \emph{continuous} functions $\phi$ and $f_{i}$. More precisely,
by applying the theory of the Glaeser refinements for bundles, he
showed that system (\ref{eq:System_1}) has a continuous solution
if and only if the affine Glaeser-stable bundle associated with system
(\ref{eq:System_1}) has no empty fiber.

Moreover J.$\,$Kollár, in the same (joint) paper \cite{FK}, starting
from the above result and making use of algebraic geometry techniques
as blowing up at singular points, proved that fixed the \emph{polynomials}
$f_{1},\ldots,f_{s}$ and assuming system (\ref{eq:System_1}) has
a solution, then:

\noindent 1) if $\phi$ is semialgebraic then there is a solution
$(\psi_{1},\ldots,\psi_{s})$ of $\phi=\underset{i}{\sum}\psi_{i}f_{i}$
such that the $\psi_{i}$ are also semialgebraic;

\noindent 2) let $U\subset\mathbb{R}^{n}\backslash Z$ (where $Z:=(f_{1}=\ldots=f_{r}=0)$)
be an open set such that $\phi$ is $C^{m}$ on $U$ for some $1\leq m\leq\infty$
or $m=\omega$. Then there is a solution $\psi=(\psi_{1},\ldots,\psi_{s})$
of $\phi=\sum\limits _{i=1}^{s}\psi_{i}f_{i}$ such that the $\psi_{i}$
are also $C^{m}$ on $U$.\medskip{}

Next, in \cite{FL-1} C. Fefferman and G.K. Luli exhibited generators
of the module $\mathcal{M}$ (over the ring of polynomials on $\mathbb{R}^{n}$)
of the vectors $f:=(f_{1},\ldots,f_{s})$ of \emph{polynomials} $f_{1},\ldots,f_{s}$
such that
\begin{equation}
\sum_{j=1}^{M}A_{ij}F_{j}=f_{i}\,\,(i=1,\ldots,N),\label{eq:ThePb}
\end{equation}
(for unknown functions $F_{1},\ldots,F_{N}\in C^{m}(\mathbb{R}^{n})$
, $m$ fixed) admits a $C^{m}$ solution.

Finally, in \cite{FL-2} C. Fefferman and G.K. Luli showed that if
$\mathcal{H}$ is a semialgebraic bundle with respect to the space
of $\mathbb{R}^{D}$-valued functions on the plane $\mathbb{R}^{2}$ with continuous
derivatives up to order $m$ (that space is called $C_{loc}^{m}(\mathbb{R}^{2};\mathbb{R}^{D})$)
and it has a $C_{loc}^{m}(\mathbb{R}^{2};\mathbb{R}^{D})$ section,
then $\mathcal{H}$ has a semialgebraic and $C_{loc}^{m}(\mathbb{R}^{2};\mathbb{R}^{D})$
section. Actually, the authors do not give an explicit method to compute
that semialgebraic $C_{loc}^{m}(\mathbb{R}^{2};\mathbb{R}^{D})$ section:
the $C_{loc}^{m}(\mathbb{R}^{2};\mathbb{R}^{D})$ semialgebraic section
is defined as the one satisfying equations (97), (98), and (99) at
p.44 of \cite{FL-2}.

In the case $m=0$ the problems of \vspace{2mm}

\begin{minipage}[t]{0.9\columnwidth}%
$-$ determining necessary and sufficient conditions for the existence
of a continuous solution of (\ref{eq:ThePb}) where $A_{ij}$ and
$f_{j}$ are given functions,

$-$ exhibiting generators of the module $\mathcal{M}$ with $A_{ij}$
given \emph{polynomials},

$-$ determining necessary and sufficient conditions for the existence
of a continuous and semialgebraic solution of (\ref{eq:ThePb}) where
$A_{ij}$ and $f_{j}$ are given \emph{polynomials} and (\ref{eq:ThePb})
admits a continuous solution,%
\end{minipage}

\vspace{3mm}

\noindent were posed by Brenner \cite{B},
and Epstein-Hochster \cite{EH}, and solved by Fefferman-Kollár \cite{FK}
and Kollár \cite{K}.

In this paper by a new approach, based on Fefferman's techniques, we generalize and solve 
the third of the above problems showed by Kollár through algebraic
techniques for $m=0$. More precisely, we
prove that if a semialgebraic bundle associated to a system with coefficients
and right-hand side that are semialgebraic (but not necessarily continuous) 
on $\mathbb{R}^{n}$ has a continuous section then it has also a continuous and semialgebraic
section. We show it without employing the algebraical blow-up theory
but only by using the analytical Fefferman-Glaeser theory with the
aim of determining an explicit method to construct a continuous and
semialgebraic section.

Let us provide a more detailed description of the problem we deal
with. We consider a \emph{semialgebraic
compact metric space} $Q\subseteq\mathbb{R}^{n}$ and a \emph{system
of linear equations 
\begin{equation}
A\left(x\right)\phi\left(x\right)=\gamma\left(x\right),\quad x=(x_{1},\ldots,x_{n})\in Q\label{eq:TargetProblem}
\end{equation}
where 
\[
Q\ni x\longmapsto A(x)=(a_{ij}(x))\in M_{r,s}(\mathbb{R})
\]
is\emph{ semialgebraic,} with $M_{r,s}(\mathbb{R})$ denoting the
set of real $r\times s$ matrices and 
\[
Q\ni x\longmapsto\gamma(x)\in\mathbb{R}^{r},\gamma(x)=\left[\begin{array}{c}
\gamma_{1}(x)\\
\vdots\\
\gamma_{r}(x)
\end{array}\right]\in\mathbb{R}^{r}
\]
being themselves\emph{ semialgebraic functions} on $Q\subseteq\mathbb{R}^{n}$.}

\ \\ \
\begin{center}
{\fboxsep 4pt\fbox{\begin{minipage}[c]{0.9\textwidth}%
Our aim is to find a \emph{necessary and sufficient condition} for
the existence of a solution $Q\ni x\longmapsto\phi\left(x\right)=\left[\begin{array}{c}
\phi_{1}(x)\\
\vdots\\
\phi_{s}(x)
\end{array}\right]\in\mathbb{R}^{s}$ of system (\ref{eq:TargetProblem}), with the $\phi_{i}:Q\rightarrow\mathbb{R}$
\emph{continuous }and \emph{semialgebraic}.%
\end{minipage}}}
\end{center}
\ \\ \ \\ \ 

We notice that the semialgebraicity of $Q$ is a necessary condition
for the existence of a semialgebraic solution of system (\ref{eq:TargetProblem})
by the definition of semialgebraic function (i.e. a function with
semialgebraic graph) and by the Tarski-Seidenberg
theorem \footnote{\textbf{Tarski-Seidenberg Theorem} $Let$ $A$ \emph{a semialgebraic
subset of} $\mathbb{R}^{n+1}$ \emph{and} $\pi:\mathbb{R}^{n+1}\rightarrow\mathbb{R}^{n}$,
\emph{the projection on the first} $n$ \emph{coordinates. Then} $\pi(A)$
\emph{is a semialgebraic subset of} $R^{n}$.

\vspace{2pt}

\textbf{Corollary}\emph{ If} $A$ \emph{is a semialgebraic subset
of} $\mathbb{R}^{n+k}$, \emph{its image by the projection on the
space of the first $n$ coordinates is a semialgebraic subset of}
$\mathbb{R}^{n}$.}. 

\vspace{0.5cm}

\noindent The plan of the paper is the following. In Section \ref{sec:The-setting}
we fix some notations and give some definitions that will be used
in Section \ref{sec:Existence}.

\noindent In Section \ref{sec:Existence} we prove that
if a semialgebraic bundle associated to a system of semialgebraic
(but not necessarily continuous) function on a semialgebraic compact
set $Q$ has a continuous section then it has also a semialgebraic
and continuous one. The main idea is to prove the result by an induction
argument on the dimension $d$ of $Q$. (We recall that the dimension
of a semialgebraic set $E\subset\mathbb{R}^{n}$ is the maximum of
the dimensions of all the embedded, not necessarily compact, submanifolds
of $\mathbb{R}^{n}$ that are contained in $E$.) In fact, for the
case $d=1$ we use the fact that a semialgebraic function on a subset
of $\mathbb{R}$ has finitely many isolated discontinuity points (the
set on which a semialgebraic function is not continuous is a semialgebraic
subset of its domain of strictly lower dimension and a semialgebraic
set of dimension $0$ is finite i.e. it is made by finitely many isolated points).
Hence, we construct a local semialgebraic and continuous section of
the bundle on a neighbourhood of each point of $Q$ and we glue the
semialgebraic and continuous sections by a semialgebraic and continuous
partition of the unity. Next, in the case $d\geq2$, by induction
hypothesis there is a continuous and semialgebraic section on an appropriate
compact subset of $Q$ of dimension $\leq d-1$ (which will be defined
in the proof of Theorem \ref{thm:Existence}) and we extend it thanks
to a semialgebraic version of Tietze-Uryshon Theorem. Finally, we
need to compute the projection of the extension on the fibers of $\mathcal{H}^{{\rm Gl}}$
(i.e. the Glaeser-stable bundle associated to the system (\ref{eq:TargetProblem}))
to obtain a continuous and semialgebraic section of $\mathcal{H}^{{\rm Gl}}$
i.e. a semialgebraic and continuous solution of (\ref{eq:TargetProblem}).

\noindent The result of Section \ref{sec:Existence} is obtained without
the use of algebraic geometrical tools, but only by the
analysis techniques such as the Glaeser refinement and the theory
of bundle sections developed by Fefferman. This paper gives an explicit
method for the construction of a semialgebraic continuous solution
of system (\ref{eq:TargetProblem}) by finitely many induction steps.

\section{The setting \label{sec:The-setting}}

Let us start by setting some notations and definitions that will be
used to pursue our goal. We shall endow every $\mathbb{R}^{s}$ used
here with euclidean norm. 
\ \\
\begin{quote}
\textbf{Notation:} Let $V\subseteq\mathbb{R}^{s}$ be an affine space
in $\mathbb{R}^{s}$ and $w\in\mathbb{R}^{s}$. We denote the \emph{projection
of $w$ on $V$} (i.e. the point $v\in V$ that makes the euclidean
norm of $v-w$ as small as possible) by $\Pi_{V}w$.
\end{quote}
\ \\
Let us consider a \emph{singular affine bundle }(or \emph{bundle}
for short) (see \cite{FK}), meaning a family $\mathcal{H}=(H_{x})_{x\in Q}$
of affine subspaces $H_{x}\subseteq\mathbb{R}^{s}$, parametrized
by the points $x\in Q$. The affine subspaces 
\[
H_{x}=\left\{ \lambda\text{\ensuremath{\in}}\mathbb{R}^{s}:A\left(x\right)\lambda=\gamma\left(x\right)\right\} ,\quad x\in Q
\]
are the \emph{fibers} of the bundle $\mathcal{H}$. (Here, we allow
the empty set $\emptyset$ and the whole space $\mathbb{R}^{s}$ as
affine subspaces of $\mathbb{R}^{s}$.)

Now we call $\mathcal{H}^{(k)}$ the $k$-th \emph{Glaeser refinement}
of $\mathcal{H}$ i.e. $\mathcal{H}^{(0)}:=\mathcal{H}$ and for all
$k\geq1$ the fibers of $\mathcal{H}^{(k)}$ are
\[
H_{x}^{(k)}:=\{\lambda\in H_{x}^{(k-1)};\,\,{\rm dist}(\lambda,H_{y}^{(k-1)})\longrightarrow0\text{ as }y\longrightarrow x\,\,\,(y\in Q)\},
\]
 for all $x\in Q$ (see Chapter 2 of \cite{FK}). We notice that $\mathcal{H}^{(k)}$
is a subbundle of $\mathcal{H}^{(k-1)}$ for all $k\geq1$. By Lemma
5 of \cite{FK} the procedure of refinement leads to a Glaeser-stable
refinement\emph{ }of $\mathcal{H}$ i.e. there is a $r\in\mathbb{N}$
such that $\mathcal{H}^{(2r+1)}=\mathcal{H}^{(2r+2)}=\cdots$. We
denote $\mathcal{H}^{(2r+1)}$ by $\mathcal{H}^{\mathrm{Gl}}$ and
we will call it the \emph{Glaeser-stable refinement} of $\mathcal{H}$
(its fibers will respectively be denoted by $H_{x}^{\mathrm{Gl}}$
for all $x\in Q$). Notice that the projection on the fibers of $\mathcal{H}^{{\rm Gl}}$
is not linear as the fibers are affine spaces and not vector spaces. 

Given a continuous solution $f$ of system (\ref{eq:TargetProblemSystem-1})
we define
\[
Q\ni y\longmapsto\omega(y):=\varPi_{\left(H_{y}^{\mathrm{Gl}}\right)^{\bot}}f(y),
\]
and we notice that $\omega$ does not depend on the choice of the
continuous solution $f$. More precisely, for all $y\in Q$ the value
$\omega(y)$ can be computed by projecting $0$ on $H_{y}^{\mathrm{Gl}}$.
Moreover, we define 
\[
Q\ni y\longmapsto\tilde{\varPi}_{1}(y)v:=\varPi_{\left(H_{y}^{{\rm Gl}}\right)^{\bot}}v,
\]
for all $v\in\mathbb{R}^{s}$. We say that $\tilde{\varPi}_{1}$ is
continuous if $y\longmapsto\tilde{\varPi}_{1}(y)e_{j}$ is continuous
for all $j=1,...,s$ with $(e_{1},...,e_{s})$ the canonical basis
of $\mathbb{R}^{s}$.

\section{Existence of a continuous semialgebraic solution
\label{sec:Existence}}

In this section we prove that system (\ref{eq:TargetProblem})
on a semialgebraic compact space $Q\subseteq\mathbb{R}^{n}$ has a
semialgebraic and continuous solution if and only if it has a continuous
one. We do it by induction on the dimension of $Q$ on which the problem
is defined.
\begin{theorem}
\label{thm:Existence} Consider a semialgebraic compact metric space\textup{
$Q\subseteq\mathbb{R}^{n}$} and a system of linear equations
\begin{equation}
A\left(x\right)\phi\left(x\right)=\gamma\left(x\right),\quad x\in Q,\label{eq:TargetProblemSystem-1}
\end{equation}
where the entries of
\[
A(x)=(a_{ij}(x_{1},\ldots,x_{n}))\in M_{r,s}(\mathbb{R})\quad\text{and}\quad\gamma(x)=\left(\gamma_{i}(x)\right)\in\mathbb{R}^{r}
\]
are themselves semialgebraic functions on $\mathbb{R}^{n}$.

\noindent Then system \emph{(\ref{eq:TargetProblemSystem-1})} has
a continuous semialgebraic solution $\phi:Q\rightarrow\mathbb{R}^{s}$
if and only if $\mathcal{H}^{\mathrm{Gl}}$ has no empty fiber.
\end{theorem}

\begin{proof} 

We start by proving the forward implication which is trivial. In fact,
if system (\ref{eq:TargetProblemSystem-1}) has a continuous solution
then $\mathcal{H}^{\mathrm{Gl}}$ has no empty fiber (see \cite{FK}).

Now, we prove the reverse implication. To do it we proceed by induction
on the dimension $d\in\{1,...,n\}$ of $Q$. (We notice that if $d=0$
then $Q$ is a finite set and, hence, any selection of $\mathcal{H}^{\mathrm{Gl}}$
is a semialgebraic and continuous section of $\mathcal{H}^{\mathrm{Gl}}$.
A selection of $\mathcal{H}^{\mathrm{Gl}}$ exists since $\mathcal{H}^{\mathrm{Gl}}$
has no empty fiber.) Actually, before starting the proof by induction,
we need to show that $\omega$ is semialgebraic. To do this we need
to verify that the set 
\[
\mathcal{H}_{Q}^{\mathrm{Gl}}:=\{(x,v)\in\mathbb{R}_{x}^{n}\times\mathbb{R}_{v}^{s};\,\,x\in Q,\,v\in H_{x}^{\mathrm{Gl}}\}\:\text{is semialgebraic.}
\]
Hence, we prove by induction on $k\geq0$ that 
\[
\mathcal{H}_{Q}^{(k)}:=\{(x,v)\in\mathbb{R}_{x}^{n}\times\mathbb{R}_{v}^{s};\,\,x\in Q,\,v\in H_{x}^{(k)}\}
\]
 is semialgebraic for all $k\geq0$. In fact,
\[
\mathcal{H}_{Q}^{(0)}=\{(x,v)\in Q\times\mathbb{R}^{s};\,\,A(x)v=\gamma(x)\}
\]
is semialgebraic and after supposing that $\mathcal{H}_{Q}^{(k-1)}$
is semialgebraic ($k\geq1)$, $\mathcal{H}_{Q}^{(k)}$ can be rewritten
as
\begin{align*}
\{(x,v) & \in\mathcal{H}_{Q}^{(k-1)};\,\,\forall\varepsilon>0,\,\exists\delta>0,\\
 & \forall(y,v^{'})\in\left(B(x,\delta)\times\mathbb{R}^{s}\right)\cap\mathcal{H}_{Q}^{(k-1)}:\,\left\Vert v-v^{'}\right\Vert <\varepsilon\},
\end{align*}
that is semialgebraic by elimination of quantifiers. Now, $\omega$
has graph given by 
\begin{equation}
\{(x,v)\in\mathcal{H}_{Q}^{\mathrm{Gl}};\,\,\nexists(x^{'},v^{'})\in\mathcal{H}_{Q}^{\mathrm{Gl}},\,x^{'}=x,\,\left\Vert v\right\Vert >\left\Vert v^{'}\right\Vert \},
\end{equation}
 with $\left\Vert \cdot\right\Vert $ the euclidean norm on $\mathbb{R}^{s}$
and, hence, it is semialgebraic by elimination of quantifiers.

In a similar way we have that if $(e_{1},...,e_{s})$ is the canonical
basis of $\mathbb{R}^{s}$ then $y\longmapsto\tilde{\varPi}_{1}(y)e_{j}$
is semialgebraic for all $j$ since its graph can be written as 
\begin{align*}
\{(y,v) & \in Q\times\mathbb{R}^{s};\,\,\exists(x^{'},v^{'})\in\left(\mathcal{H}^{\mathrm{Gl}}-\omega\right)_{Q},\,x^{'}=y,\,\left\langle v^{'},v\right\rangle =0,\,v+v^{'}=e_{j}\},
\end{align*}
which is semialgebraic by elimination of quantifiers because
\[
\left(\mathcal{H}^{\mathrm{Gl}}-\omega\right)_{Q}:=\{(x,v)\in Q\times\mathbb{R}^{s};\,\,\exists(x^{'},v^{'})\in\mathcal{H}_{Q}^{\mathrm{Gl}}-\omega,\,x^{'}=x,\,v=v^{'}-\omega(x)\}
\]
is semialgebraic again by elimination of quantifiers.

Now, we are ready to begin the proof by induction. 

We start with the case $d=1$. For every given $x\in Q$ there exists
$v_{x}\in H_{x}^{\mathrm{Gl}}\neq\emptyset$ and a ball $B(x,r_{v_{x}})\subseteq\mathbb{R}^{n}$
such that
\begin{align*}
Q\cap B(x,r_{v_{x}})\ni y\longmapsto\tilde{\gamma}_{v_{x}}(y):= & \varPi_{H_{y}^{\mathrm{Gl}}}v_{x}\\
= & \omega(y)+v_{x}-\tilde{\varPi}_{1}(y)v_{x}
\end{align*}
 is semialgebraic since $\omega$ and $\tilde{\varPi}_{1}$ are semialgebraic.
We show that $\tilde{\gamma}_{v_{x}}$ is continuous for $r_{v_{x}}$
small enough. In fact, if we suppose by
contradiction that there is no $r_{v_{x}}$ such that $\tilde{\gamma}_{v_{x}}$
is continuous then for all $n\in\mathbb{N}$
there is $y_{n}\in B(x,\frac{r_{v_{x}}}{n})$ such that $\tilde{\gamma}_{v_{x}}$
is discontinuous at $y_{n}$. Hence, there are two possibilities:

1. $\forall n\in\mathbb{N},\,y_{n}\neq x$. A semialgebraic function
is real analytic on the complementary of a semiagebraic set of dimension
strictly less than the one of its domain. (In fact, the domain of
a semialgebraic function is semialgebraic by the Tarski-Seidenberg
theorem.) Thus, since $\gamma_{v_{x}}$ is semialgebraic on $Q$
the discontinuity points of $\gamma_{v_{x}}$ are finitely many. Hence,
we come to a contradiction;

2. $\exists\overline{n}\in\mathbb{N}\text{ such that }y_{\overline{n}}=x$.
Now, since $v_{x}\in H_{x}^{\mathrm{Gl}}$ on the one hand 
\[
\mathrm{dist}(v_{x};H_{y}^{\mathrm{Gl}})\underset{{\scriptstyle y\rightarrow x}}{\longrightarrow}0,
\]
and, on the other, 
\[
\left\Vert \Pi_{H_{y}^{{\rm Gl}}}v_{x}-v_{x}\right\Vert ={\rm dist}(v_{x},H_{y}^{{\rm Gl}}).
\]
Therefore 
\[
\Pi_{H_{y}^{{\rm Gl}}}v_{x}\underset{{\scriptstyle y\rightarrow x}}{\longrightarrow}v_{x}\underset{\overset{\uparrow}{{\scriptstyle v_{x}\in H_{x}^{\mathrm{Gl}}}}}{=}\Pi_{H_{x}^{{\rm Gl}}}v_{x}.
\]
This is impossible since $\gamma_{v_{x}}$ would be continuous at
$x$, contrary to the assumption. Thus $\tilde{\gamma}_{v_{x}}$
is continuous upon possibly reducing the ball radius $r_{v_{x}}$.

Now we glue these local solutions thanks to a semialgebraic continuous
partition of the unity. In fact, we notice that the set of balls $\{B(x,\overline{r}_{v_{x}})\}_{x\in Q}$,
where $v_{x}$ is chosen in $H_{x}^{\mathrm{Gl}}$, is an open cover
of the compact space $Q$. Then there is $N$ such that $\{B(x_{i},\overline{r}_{v_{x_{i}}})\}_{i=1,\ldots,N}$
is an open cover of $Q$. Consider 
\begin{equation}
\mu_{(x,r)}(y):=\begin{cases}
\sqrt{r^{2}-\left\Vert y-x\right\Vert ^{2}} & \text{for }y\in B(x,r),\\
0 & \text{for }y\notin B(x,r).
\end{cases}\label{eq:PartitionUnity}
\end{equation}
Notice that $\mu_{(x,r)}(y)$ is semialgebraic and continuous on $Q$,
$\forall x\in Q$, $\forall r\in\mathbb{R}^{+}$ and that ${\displaystyle \sum_{i=1}^{N}}\mu_{(x_{i},\overline{r}_{v_{x_{i}}})}(y)>0$
for each $y\in Q$ as $\mu_{(x,r)}(y)\geq0$ for every $y\in Q$ and
$\mu_{(x,r)}(y)>0$ for all $y\in B(x,r)$. Moreover, for all $y\in Q$
there is $B(x_{i},\overline{r}_{v_{x_{i}}})$ as above such that $y\in B(x_{i},\overline{r}_{v_{x_{i}}})$
since $\{B(x_{i},\overline{r}_{v_{x_{i}}})\}_{i=1,\ldots,N}$ is an
open covering of $Q$. Hence the function 
\[
Q\ni y\longmapsto\phi(y):=\frac{1}{{\displaystyle \sum_{i=1}^{N}}\mu_{(x_{i},\overline{r}_{v_{x_{i}}})}(y)}\sum_{j=1}^{N}\mu_{(x_{j},\overline{r}_{v_{x_{j}}})}(y)\Pi_{H_{y}^{{\rm Gl}}}v_{x_{j}}
\]
is a semialgebraic and continuous solution of the system on $Q$.
(We also notice that $\phi(y)\in H_{y}^{{\rm Gl}}$ for all $y\in Q$.)

Next, we suppose that $Q$ is a semialgebraic subset of dimension
$\tilde{d}\leq n$ and that we can write a semialgebraic and continuous
section of $\mathcal{H}^{\mathrm{Gl}}$ on any compact semialgebraic
subset of $Q$ of dimension $d\leq\tilde{d}-1$ of $Q$. We want to
construct a semialgebraic and continunous section of $\mathcal{H}^{\mathrm{Gl}}$
on $Q$. We will call $U$ the subset of $Q$ where $\omega$ or $\tilde{\varPi}_{1}$
is not continuous. Since we proved that $\omega$ and $y\longmapsto\tilde{\varPi}_{1}(y)e_{j}$
are semialgebraic (for all $j$), $U$ is a semialgebraic set of dimension
$\leq\tilde{d}-1$. (A zero-dimensional semialgebraic subset of $\mathbb{R}^{n}$
is finite. A one-dimensional semialgebraic subset of $\mathbb{R}^{n}$
is a union of finitely many real-analytic arcs and finitely many points.
See Chapter 2 of \cite{BCR}.) 

Thus, by inductive hypothesis there is a semialgebraic and continuous
section $S$ of $\mathcal{H}^{\mathrm{Gl}}$ on $\overline{U}$. In
fact, $\overline{U}$ is a compact semialgebraic subset of $Q$ of
dimension $\tilde{d}-1$. Now, $S$ can be extended to a semialgebraic
and continuous function on $Q$ by Proposition 2.6.9 at p. 45 of \cite{BCR}
which is a semialgebraic version of Tietze-Uryshon Theorem and we
will call $S$ that extension again. Actually, $S$ is defined on
$Q$, but it is a section of $\mathcal{H}^{\mathrm{Gl}}$ only on
$\overline{U}$. Hence, we compute the projection of $S$ on the fibers
of $\mathcal{H}^{\mathrm{Gl}}$ i.e. 

\begin{align*}
Q\ni y\longmapsto\sigma(y):= & \varPi_{H_{y}^{\mathrm{Gl}}}S(y)\\
= & \omega(y)+S(y)-\tilde{\varPi}_{1}(y)S(y).
\end{align*}

\noindent We notice that $\sigma$ is semialgebraic and also that
$\sigma$ is continuous on $Q\setminus\overline{U}$ since $\omega$
and $y\longmapsto\tilde{\varPi}_{1}(y)S(y)$ is continuous on $Q\setminus\overline{U}$.
Moreover, $\sigma$ is continuous on $\overline{U}$ since $S(x)\in H_{x}^{\mathrm{Gl}}$
for all $x\in\overline{U}$ and, hence, we can proceed as done to
prove that $x$ is not a discontinuity point for $\tilde{\gamma}_{v_{x}}$
in the case $d=1$. In fact, for all $x\in\overline{U}$ and all $y\in Q$
\begin{align*}
\left\Vert \Pi_{H_{y}^{{\rm Gl}}}S(y)-\underbrace{\Pi_{H_{x}^{{\rm Gl}}}S(x)}_{=S(x)\in H_{x}^{\mathrm{Gl}}}\right\Vert \leq & \left\Vert \Pi_{H_{y}^{{\rm Gl}}}S(y)-S(y)\right\Vert +\left\Vert S(y)-S(x)\right\Vert \\
\leq & \left\Vert \Pi_{H_{y}^{{\rm Gl}}}S(x)-S(y)\right\Vert +\left\Vert S(y)-S(x)\right\Vert \\
\leq & \left\Vert \Pi_{H_{y}^{{\rm Gl}}}S(x)-S(x)\right\Vert +2\left\Vert S(y)-S(x)\right\Vert ,
\end{align*}
where the second inequality follows from the minimal distance property
of the projection (we notice that $\Pi_{H_{y}^{{\rm Gl}}}S(x)\in H_{y}^{{\rm Gl}}$).
Now, $\left\Vert S(y)-S(x)\right\Vert \underset{y\rightarrow x}{\longrightarrow}0$
by the continuity of $S$ and $\left\Vert \Pi_{H_{y}^{{\rm Gl}}}S(x)-S(x)\right\Vert ={\rm dist}(S(x),H_{y}^{{\rm Gl}})\underset{y\rightarrow x}{\longrightarrow}0$
since $S(x)\in H_{x}^{\mathrm{Gl}}$.

The proof is complete.
\end{proof}
\begin{remark*}
Since the absence of empty fiber of $\mathcal{H}^{\mathrm{Gl}}$ is
equivalent to the existence of a continuous section of $\mathcal{H}^{\mathrm{Gl}}$
(see \cite{FK}) and, hence, of $\mathcal{H}$, we have just proved
that system (\ref{eq:TargetProblemSystem-1}) on a semialgebraic compact
space $Q\subseteq\mathbb{R}^{n}$ has a semialgebraic and continuous
solution if and only if it has a continuous one.
\end{remark*}
\section*{Acknowledgements}
I wish to thank Alberto Parmeggiani for useful discussions.



\end{document}